\documentclass[a4paper, 11pt]{amsart}

\usepackage[T1]{fontenc}
\usepackage[utf8]{inputenc}

\usepackage{amssymb}

\makeatletter
\@namedef{subjclassname@2020}{%
  \textup{2020} Mathematics Subject Classification}
\makeatother

\allowdisplaybreaks[2]

\newtheorem{theorem}{Theorem}[section]

\newtheorem{lemma}[theorem]{Lemma}

\theoremstyle{definition}

\newtheorem{example}[theorem]{Example}

\numberwithin{equation}{section}

\newcommand\A{{\mathcal A}}

\newcommand\C{\mathbb{C}}

\begin{document}

\title[The Waring Problem for Matrix Algebras, II]{The Waring Problem for Matrix Algebras, II}


\author{Matej Brešar${}^{1,2,3}$} 

\author{Peter \v Semrl${}^{1,3}$}

\address{${}^{1}$Faculty of Mathematics and Physics,  University of Ljubljana,  Slovenia}
\address{${}^{2}$Faculty of Natural Sciences and Mathematics, University
of Maribor, Slovenia}

\address{${}^{3}$Institute of Mathematics, Physics and Mechanics, Ljubljana, Slovenia}

\email{matej.bresar@fmf.uni-lj.si}

\email{peter.semrl@fmf.uni-lj.si}

\thanks{Supported by ARRS Grants P1-0288 and J1-2454.}



\maketitle

\begin{abstract}
    Let $f$
    be a noncommutative 
    polynomial of degree $m\ge 1$ over 
     an algebraically closed field $F$ of characteristic $0$. 
     If $n\ge m-1$ and 
     $\alpha_1,\alpha_2,\alpha_3$ are nonzero elements from $F$ such that $\alpha_1+\alpha_2+\alpha_3=0$,  then every trace zero $n\times n$ matrix over $F$ can be written as
$\alpha_1 A_1+\alpha_2A_2+\alpha_3A_3$ for some $A_i$ in the image of $f$ in $M_n(F)$.
\end{abstract}
\section{Introduction}

In 2011, Larsen, Shalev, and Tiep 
\cite{LST} proved  that given a word  $w=w(x_1,\dots,x_m) \ne 1$,  every element in any
 finite non-abelian simple group $G$ of sufficiently high order can be written as the product of two elements from
$w(G)$,
  the image of the word map 
 induced by $w$. This is  a definitive solution of
 the Waring problem for finite simple groups 
 which covers several partial solutions obtained earlier. In particular, in 2009 Shalev \cite{Sh}  proved that, under the same assumptions, every element in $G$ is the product of three elements from
$w(G)$.

In the recent papers \cite{B, BS}, the authors initiated the study of various
 similar Waring type problems for algebras.  The present short paper is focused on the one  which seems particularly close to the group-theoretic Waring problem. It was stated as Question 4.8 in \cite{BS} and asks the following: Given a nonconstant noncommutative polynomial $f$  with coefficients from a field $F$, is it then true that, for any sufficiently large $n$, every traceless $n\times n$ matrix $T$ over $F$ is the difference of two elements from $f(M_n(F))$, the image of $f$
in the matrix algebra $M_n(F)$?
Unfortunately, we will not give a definitive answer, 
but will make a step towards the solution.
Namely, 
under the assumptions that 
$F$ is an algebraically closed field of characteristic $0$ and
the degree of $f$ does not exceed $n+1$, we will prove that $T$ is a linear combination of three elements from $f(M_n(F))$; in fact,
for any nonzero elements
$\alpha_1,\alpha_2,\alpha_3$ from $F$ satisfying $\alpha_1+\alpha_2+\alpha_3=0$, 
there exist $A_1,A_2,A_3\in f (M_n(F))$ such that  $$T=
\alpha_1 A_1+\alpha_2A_2+\alpha_3A_3$$ (for example, $T$ is equal to
$A_1+A_2-2A_3$ for some 
$A_i\in f(M_n(F))$).

The proof will be given in Section 3 where we will also provide all the necessary definitions and other information needed for understanding the problem.
Section 2 is devoted to a linear algebraic result that will be used in the proof, but is  of independent interest. The other two ingredients of the proof are a result from \cite{BS} on  images of polynomials in $M_p(F)$ with $p$  prime and Bertrand's postulate.

\section{A linear algebraic result}

We denote by $M_n(F)$ the algebra of all $n \times n$  matrices over the field $F$ and by sl$_n(F) $ the subspace of all trace zero matrices in $M_n(F)$. The goal of this section is to prove the following theorem.

\begin{theorem}\label{glavni}
Let $F$
 be a field of characteristic $0$, let $n$ and $q$ be positive integers such that $n \ge 2$ and $\frac{n}{2} \le q \le n$, let $\lambda_1 , \ldots , \lambda_q$ be distinct nonzero elements from $F$, and let $D\in M_n(F)$ be  the diagonal  matrix whose first $q$ diagonal entries are $\lambda_1 , \ldots , \lambda_q$ and all other diagonal entries are zero. Let
 $\alpha_1,\alpha_2,\alpha_3$ be nonzero elements from $F$ such that $\alpha_1+\alpha_2+\alpha_3=0$.
 Then every $T \in {\rm sl}_n(F)$ can be written as $$T = \alpha_1 A_1 + \alpha_2 A_2 +\alpha_3A_3$$
 for some matrices $A_1,A_2,A_3$ that are similar to $D$.
\end{theorem}

For the proof we  need several lemmas. In what follows, $F$, $n$, $q$,  $\lambda_1 , \ldots , \lambda_q$, and $D$ will be as in Theorem \ref{glavni}.
We  further denote by $D_q$ the $q \times q$ diagonal matrix with the diagonal entries $\lambda_1 , \ldots , \lambda_q$.  As usual, $I_m$ stands for the identity $m\times m$ matrix.

The first two lemmas
follow easily from the fact that triangular matrices with distinct diagonal entries  are 
diagonalizable, so we omit the proofs.


\begin{lemma}\label{bubu}
Let $L$ be a lower triangular $q \times q$ matrix with the diagonal entries $\lambda_1 , \ldots , \lambda_q$ and $W$ any $(n-q) \times q$ matrix. Then the matrix
$$
 \left[ \begin{matrix} L & 0 \cr W  & 0 \cr \end{matrix} \right]
$$
is similar to $D$.
\end{lemma}


\begin{lemma}\label{bubum}
Let $M$ be an upper triangular $q \times q$ matrix with the diagonal entries $\lambda_1 , \ldots , \lambda_q$ and $V$ any $q\times (n-q)$ matrix. Then the matrix
$$
 \left[ \begin{matrix} M & V \cr 0  & 0 \cr \end{matrix} \right]
$$
is similar to $D$.
\end{lemma}


\begin{lemma}\label{babam}
 Let $k\ge 2$ and let
 $Z$  be a nonscalar $k \times k$ matrix. Then there exist $k \times k$ matrices $W$ and  $X$ such that $WX=Z$ and  $XW$ is not a
diagonal matrix.
\end{lemma}

\begin{proof} All we need to do is to find an invertible matrix $X$ such that $XZX^{-1}$ is not diagonal. Indeed,
take any invertible $k \times k$ matrix $X$ and set $W = ZX^{-1}$. Then $WX = Z$ and $XW = XZX^{-1}$.

We may assume that
 $Z$ is diagonal. As $Z$ is nonscalar, we may also assume with no loss of generality  that the first and the second diagonal entry of $Z$ are not equal. Considering invertible matrices $X$ of the form
$$
X = \left[ \begin{matrix} X_1 & 0 \cr 0 & I_{k-2} \end{matrix} \right]  
$$
and the upper-left $2 \times 2$ corner of the matrix $XZX^{-1}$ we reduce the problem to the case where $Z$ is a $2 \times 2$ nonscalar diagonal matrix, which is an easy exercise.
\end{proof}

\begin{lemma}\label{bebem}
Let $k,l$ be positive integers with $l > k$ and $Z$ any $k \times k$ matrix. Then there exist a $k \times l$ matrix $W$ and an $l \times k$ matrix $X$ such that $WX=Z$ and the $l \times l$ matrix $XW$ is not 
diagonal. 
\end{lemma}

\begin{proof}
Take any nonzero  $k \times (l-k)$ matrix $Y$ and set
$$
W = \left[ \begin{matrix} Z & Y \cr \end{matrix} \right]  \ \ \ {\rm and} \ \ \  X = \left[ \begin{matrix} I_k \cr 0 \cr \end{matrix} \right] .
$$
Then clearly, 
$WX=Z$ and $XW$ is not 
diagonal. 
\end{proof}

\begin{lemma}\label{prva}
Let $n > 2$ and let $Z$ be an $(n-q) \times (n-q)$ matrix. If $n$ is even and $q = \frac{n}{2}$ we additionally assume that $Z$ is not a scalar matrix.
Then there exist a nonscalar $q\times q$ matrix $U$, a $q \times (n-q)$ matrix $V$, and an $(n-q) \times q$ matrix $W$ such that
$$
\left[ \begin{matrix} U & V \cr W  & Z \cr \end{matrix} \right]
$$
is similar to $D$.
\end{lemma}

\begin{proof}
The statement is trivial when $n=q$, so assume that $q < n$.

If $q > n-q$ we can apply Lemma \ref{bebem} to find an $(n-q) \times q$ matrix $W$ and a $q \times (n-q)$ matrix $X$ such that
$
WX = Z
$
and $XW$ is not diagonal. In the remaining case where $q= n-q$  the existence of such matrices $W$ and $X$ follows from Lemma \ref{babam}.
 
Set 
$$
U = D_q - XW \ \ \ {\rm and} \ \ \  V = UX + XWX - XZ.
$$
It follows that $U$ is not diagonal, and hence $U$ is not a scalar matrix. Moreover,
$$
\left[ \begin{matrix} I_q & X \cr 0  & I_{n-q} \cr \end{matrix} \right] \, \left[ \begin{matrix} U & V \cr W  & Z \cr \end{matrix} \right] \, 
\left[ \begin{matrix} I_q & -X \cr 0  & I_{n-q} \cr \end{matrix} \right]
= \left[ \begin{matrix} D_q & 0 \cr W  & 0 \cr \end{matrix} \right],
$$
which together with Lemma \ref{bubu} completes the proof.
\end{proof}

\begin{lemma}\label{bibim}
If a $q \times q$ matrix $T$ is not a scalar matrix, then there exists an invertible $q \times q$ matrix $R$ such that $RTR^{-1} - T$ is not diagonal. 
\end{lemma}

\begin{proof} 
Observe first that if $P$ is a $q\times q$ permutation matrix then a $q\times q$ matrix $W$ is diagonal if and only if $PWP^{-1}$ is diagonal. Since $$P(RTR^{-1} - T)P^{-1} = (PRP^{-1}) PTP^{-1}  (PRP^{-1})^{-1} - PTP^{-1}$$ we may replace $T$ by $PTP^{-1}$, where $P$ is any permutation matrix. As $T$ is not a scalar matrix, the upper-left $2\times 2$ corner of $PTP^{-1}$ is not scalar for some permutation matrix $P$. 
Then considering invertible matrices
$$
R = \left[ \begin{matrix} R_1 & 0 \cr 0 & I_{q-2} \end{matrix} \right]  
$$
we reduce the general problem to the $2\times 2$ case, which is an easy exercise.
\end{proof}

The next  lemma may be interesting in its own right.
\begin{lemma}\label{seto}
Let $k$ be a positive integer, $B$  any nonscalar $k\times k$ matrix, and $\mu_1 , \ldots , \mu_k$ elements from  $F$ such that 
$$
{\rm tr}\, B = \mu_1 + \ldots + \mu_k .
$$
Then $B$ is similar to
 a $k \times k$ matrix  whose 
  diagonal entries  are $\mu_1 , \ldots , \mu_k$.
\end{lemma}

\begin{proof}
Throughout the proof, we  identify $k\times k$ matrices with linear operators on $F^k$.

In the first step, we will show  that there exists a $k \times k$ matrix $C$ similar to $B$ such that its $(1,1)$-entry  is equal to $\mu_1$. 
Since $B$ is not a scalar matrix there exists a vector $x \in F^k$  such that $Bx$ and $x$ are linearly independent. It follows that we can find a linear functional $f : F^k \to F$ such that $f(x) = 1$ and $f(Bx) = \mu_1$. We apply the fact that the kernel of $f$ is of codimension one to see that there exists a nonzero vector $y \in {\rm Ker}\, f \cap {\rm span}\, \{ x, Bx \}$. From $f(x) = 1$ and $f(y) = 0$ we infer that $x$ and $y$ are linearly independent. Therefore,
$$
Bx = \alpha x + \beta y
$$
for some $\alpha, \beta \in F$. Applying $f$ to the above equality we arrive at $$\alpha = f(Bx) = \mu_1.$$ Hence, if we extend the linearly independent pair of vectors $x,y$ to a basis of $F^k$ and represent the operator $B$ in this basis, then the first diagonal entry of the obtained matrix is equal to $\mu_1$.

We are now ready to prove the statement of the lemma. There are no nonscalar matrices if $k=1$. 
The first step implies that our statement is true for $k=2$.

We proceed by induction. Let $k \ge 3$ and assume that the statement is true for all matrices of the size $(k-1) \times (k-1)$. Let $B$ be a 
$k\times k$ matrix and $\mu_1 , \ldots , \mu_k$ elements from $F$ such that 
${\rm tr}\, B = \mu_1 + \ldots + \mu_k$. We already know that $B$ is similar to the matrix
$$
C = \left[ \begin{matrix} \mu_1 & u \cr v  & C_{1} \cr \end{matrix} \right],
$$
where $C_1$ is a $(k-1) \times (k-1)$ matrix. Clearly,
$$
{\rm tr}\, C_1 = \mu_2 + \ldots + \mu_k .
$$

We distinguish two possibilities. The easier one is when $C_1$ is not a scalar matrix. Then 
by the induction hypothesis there exists an invertible $(k-1) \times (k-1)$ matrix $S$ such that $SC_1 S^{-1}$ is a matrix whose diagonal entries are $\mu_2 , \ldots, \mu_k$.
Therefore, 
$$
 \left[ \begin{matrix} 1 & 0 \cr 0  & S \cr \end{matrix} \right] \,  \left[ \begin{matrix} \mu_1 & u \cr v  & C_{1} \cr \end{matrix} \right] \,  \left[ \begin{matrix} 1 & 0 \cr 0  & S^{-1} \cr \end{matrix} \right] 
= \left[ \begin{matrix} \mu_1 & uS^{-1} \cr Sv  & SC_{1}S^{-1} \cr \end{matrix} \right],
$$
which gives the desired conclusion.

It remains to consider the case when $B$ is similar to the matrix
$$
C = \left[ \begin{matrix} \mu_1 & u \cr v  & \lambda I_{k-1} \cr \end{matrix} \right],
$$
where 
$$
\lambda = (k-1)^{-1} (\mu_2 + \ldots + \mu_k ).
$$
We again distinguish two cases. Let us first assume that $u \not=0$ or $v\not=0$. We will consider just the case when $v\not=0$ since the proof in the case when $u \not=0$ goes through in an almost the same way. Since $v$ is a nonzero $(k-1) \times 1$ column vector and $k-1 \ge 2$ we can find an $1 \times (k-1)$ vector $a$ such that 
$$
av= 0 \ \ \ {\rm and} \ \ \ va \not = 0.
$$
We have
$$
\left[ \begin{matrix} 1 & a \cr 0  &  I_{k-1} \cr \end{matrix} \right] \, \left[ \begin{matrix} \mu_1 & u \cr v  & \lambda I_{k-1} \cr \end{matrix} \right] \, \left[ \begin{matrix} 1 & -a \cr 0  &  I_{k-1} \cr \end{matrix} \right] 
=  \left[ \begin{matrix} \mu_1 & u + (\lambda - \mu_1) a \cr v  & \lambda I_{k-1} - va \cr \end{matrix} \right] .
$$
Since $k-1 \ge 2$ and $va$ is a nonzero rank one matrix, the matrix $ \lambda I_{k-1} - va$ is nonscalar and we can complete the induction step exactly in the same way as in the previous case.

It remains to consider the case when $B$ is similar to the matrix
$$
C = \left[ \begin{matrix} \mu_1 & 0 \cr 0  & \lambda I_{k-1} \cr \end{matrix} \right],
$$
where 
$$
\lambda = (k-1)^{-1} (\mu_2 + \ldots + \mu_k ) \not = \mu_1.
$$
If $\mu_2 = \ldots = \mu_k = \lambda$, then we are done. So, assume this is not true. Then clearly, at least one of the elements $\mu_2, \ldots , \mu_k$ is not contained in the set $\{ \mu_1 , \lambda \}$. We will consider just one of the cases, say the case when 
$\mu_2 \not\in \{ \mu_1 , \lambda \}$. Then by the first step of the proof the matrix $C$ is similar to the matrix
$$
D = \left[ \begin{matrix} \mu_2 & u \cr v  & D_{1} \cr \end{matrix} \right],
$$
where $D_1$ is a $(k-1) \times (k-1)$ matrix. Obviously,
$$
{\rm tr}\, D_1 = \mu_1 + \mu_3 + \ldots + \mu_k .
$$
It is clear that the possibility that $D_1$ is a scalar matrix and $u=0$ and $v=0$ cannot occur since  $\mu_2$ would then be an eigenvalue of $C$, a contradiction. Hence, either $D_1$ is not a scalar matrix or $D_1$ is a scalar matrix and not both of $u$ and $v$ are zero. In either case we apply the induction hypothesis in the same way as above to conclude that $B$ is similar to
 a matrix  whose 
  diagonal entries  are $\mu_2 , \mu_1 , \mu_3 , \ldots , \mu_k$. Applying a suitable permutation similarity we finally  conclude that $B$ is similar to
 a $k \times k$ matrix  whose 
  diagonal entries  are $\mu_1 ,\mu_2, \ldots , \mu_k$.
\end{proof}

\bigskip

\noindent
{\em Proof of Theorem \ref{glavni}.}
Let $T$ be an $n\times n$ trace zero matrix. The result is trivial if  $T=0$, so we  assume that $T \not=0$. In particular, $T$ is not a scalar matrix and $n \ge 2$. The special case when $n=2$ will be treated at the end of the proof. So, assume that $n \ge 3$. It is also clear that the desired conclusion holds for a given $T$ if and only it holds for any matrix that is similar to $T$. Clearly, the fact that $T$ is nonscalar implies that $T$ is similar to a matrix whose bottom-right $2 \times 2$ corner is not scalar. Therefore, there is no loss of generality in assuming that $T$ is represented in a block form
$$
T = \left[ \begin{matrix} T_1 & T_2 \cr T_3  & T_4 \cr \end{matrix} \right],
$$
where the sizes of $T_1$ and $T_4$ are $q \times q$ and $(n-q) \times (n-q)$, respectively, and either $n-q < q$ or  $n=2q$ and $T_4$ is not a scalar matrix. By Lemma \ref{prva} there exist matrices $U',V',W'$ of the appropriate sizes such that $U'$ is not  a scalar matrix and
$$
A'_1= \left[ \begin{matrix} U' & V' \cr W'  & \alpha_1^{-1}T_4 \cr \end{matrix} \right]
$$
is similar to $D$. If $T_1 - \alpha_1 U'$ is not a scalar matrix set
$$
A_1 = A'_1, \ \ \ U = U', \ \ \ V = V', \ \ \ {\rm and} \ \ \ W = W'.
$$
If $T_1 - \alpha_1 U'$ is a scalar matrix then, because $U'$ is not a scalar matrix, we can use Lemma \ref{bibim} to find an invertible $q \times q$ matrix $R$ such that $T_1 - \alpha_1 R U' R^{-1}$ is not a scalar matrix. In this case we set
$$
U = RU' R^{-1}, \ \ \ V = RV', \ \ \ {\rm and} \ \ \ W = W' R^{-1},
$$
and
\begin{align*}
A_1&= \left[ \begin{matrix} U & V \cr W  & \alpha_1^{-1}T_4 \cr \end{matrix} \right] = \left[ \begin{matrix} RU' R^{-1} & RV' \cr W' R^{-1}  & \alpha_1^{-1}T_4 \cr \end{matrix} \right]\\
&= \left[ \begin{matrix} R & 0 \cr 0  & I_{n-q} \cr \end{matrix} \right] \cdot A_1'\cdot
\left[ \begin{matrix} R^{-1} & 0 \cr 0  & I_{n-q} \cr \end{matrix} \right] .
\end{align*}

Since ${\rm tr}\, A_1 = {\rm tr}\, D = \lambda_1 + \ldots + \lambda_q$ and ${\rm tr}\, T = 0$, we have
$$
T-\alpha_1 A_1 = \left[ \begin{matrix} T_1 - \alpha_1 U& T_2 -\alpha_1 V \cr T_3 -\alpha_1 W  & 0 \cr \end{matrix} \right]
= \left[ \begin{matrix} S_1 & S_2 \cr S_3  & 0 \cr \end{matrix} \right]
$$
with $${\rm tr} \, S_1 = - \alpha_1\lambda_1 - \ldots - \alpha_1 \lambda_q.$$ Moreover, $S_1$ is not a scalar matrix.

We need to show that  there exist $n\times n$ matrices $A_2,A_3$ similar to $D$ such that
$
T-\alpha_1 A_1 = \alpha_2 A_2+\alpha_3A_3.
$
By Lemma \ref{seto}, we see that after applying an appropriate similarity transformation we may assume with no loss of generality that the diagonal entries of $S_1$ are $-\alpha_1 \lambda_1, \ldots , -\alpha_1 \lambda_q$. Let $L$ be the lower triangular $q\times q$ matrix with diagonal entries $\lambda_1 , \ldots , \lambda_q$ whose strictly lower triangular part coincides with the strictly lower triangular part of $\alpha_2^{-1}S_1$. Let further   $M$ be the upper triangular $q\times q$ matrix with diagonal entries $\lambda_1 , \ldots , \lambda_q$ whose strictly upper triangular part coincides with the strictly upper triangular part of $\alpha_3^{-1} S_1$. Set
$$
A_2 = \left[ \begin{matrix} L & 0 \cr \alpha_2^{-1}S_3  & 0 \cr \end{matrix} \right] \ \ \ {\rm and} \ \ \ A_3 =  \left[ \begin{matrix} M & \alpha_3^{-1} S_2 \cr 0  & 0 \cr \end{matrix} \right].
$$
Then  $A_2$ is similar to $D$ by Lemma  \ref{bubu}, $A_3$  is similar to $D$ by Lemma  \ref{bubum}, and 
using 
$\alpha_1+\alpha_2 + \alpha_3 =0$ we see that
$T-\alpha_1 A_1 = \alpha_2 A_2+\alpha_3A_3$.
 
It remains to consider the case when $n=2$. Then
$$
D = \left[ \begin{matrix} \lambda_1 & 0 \cr 0 & \lambda_2 \cr \end{matrix} \right]  = (\lambda _ 1 - \lambda_2 ) \left[ \begin{matrix} 1 & 0 \cr 0  & 0\cr \end{matrix} \right] + \lambda_2 I_2
$$
with $\lambda_1 \not = 0$ and $\lambda_1 \not= \lambda_2$. Clearly, a matrix  $T \in {\rm sl}_2(F)$ can be written as $$T = \alpha_1 A_1 + \alpha_2 A_2 +\alpha_3A_3$$
 for some matrices $A_1,A_2,A_3$ that are similar to $D$ if and only if $T$ can be written as $$T = \alpha_1 (\lambda _ 1 - \lambda_2 ) B_1 + \alpha_2 (\lambda _ 1 - \lambda_2) B_2 +\alpha_3 (\lambda _ 1 - \lambda_2) B_3$$
 for some matrices $B_1,B_2,B_3$ that are similar to
$
 \left[ \begin{smallmatrix} 1 & 0 \cr 0  & 0\cr \end{smallmatrix} \right].
$
Hence, with no loss of generality we can assume that $D =  \left[ \begin{smallmatrix} 1 & 0 \cr 0  & 0\cr \end{smallmatrix} \right]$.

Let  $\alpha_1,\alpha_2,\alpha_3$ be nonzero elements from $F$ such that $\alpha_1+\alpha_2+\alpha_3=0$ and  $T \in {\rm sl}_2(F)$ a nonscalar matrix. By Lemma \ref{seto}, $T$ is similar to
$$
\left[ \begin{matrix} { \alpha_1 - \alpha_2 \over 2} & \gamma \cr \delta &  { \alpha_2 - \alpha_1 \over 2} \cr \end{matrix} \right]
$$
for some $\gamma , \delta \in F$. It is easy to determine $a,b \in F$ such that
$$
\left[ \begin{matrix} { \alpha_1 - \alpha_2 \over 2} & \gamma \cr \delta &  { \alpha_2 - \alpha_1 \over 2} \cr \end{matrix} \right] = \alpha_1  \left[ \begin{matrix} 1 & a \cr 0  & 0\cr \end{matrix} \right] 
+ \alpha_2  \left[ \begin{matrix} 0 & 0 \cr b  & 1 \cr \end{matrix} \right] - (\alpha_1 + \alpha_2 )  \left[ \begin{matrix} \frac{1}{2} & \frac{1}{2} \cr \frac{1}{2}  & \frac{1}{2}\cr \end{matrix} \right] .
$$
This completes the proof.
\hfill\(\Box\)

\bigskip

We conclude this section by showing that Theorem \ref{glavni} cannot be improved to state that
$T$ is a linear combination of two matrices similar to $D$.

\begin{example}
Let $n\ge 6$ and $\frac{n}{2}\le q < n-2$.
Pick an idempotent $E$ of rank one and set 
$T=I_n-nE$. Observe that $T$ has trace zero. Suppose $T=\alpha_1 A_1 + \alpha_2 A_2$
for some scalars $\alpha_1, \alpha_2$
and  matrices $A_1$ and $A_2$ that are similar to $D$.
Thus, $$\alpha_2 A_2+  nE = I_n-\alpha_1 A_1.$$ Since 
$I_n-\alpha_1 A_1$ is similar to $I_n-\alpha_1 D$, it has rank at least $n-1$. On the other hand, the subadditivity  of rank implies that $\alpha_2 A_2 + nE$ has rank at most $q+1 < n-1$, a contradiction.
\end{example}

\section{Main theorem}

 Let $F\langle \mathcal X\rangle$ denote the free algebra generated by the set $\mathcal X = \{X_1,X_2,\dots\}$ over the field $F$.
 The elements of $F\langle \mathcal X\rangle$ are called {\em noncommutative polynomials}.
If $f=f(X_1,\dots,X_m)\in F\langle \mathcal X  \rangle$ is a noncommutative polynomial and $\A$ is  an $F$-algebra, we call the set
$$f(\A)=\{f(a_1,\dots,a_m)\,|\,a_1,\dots,a_m\in \A\}$$
 the {\em image of $f$} in $\A$. We refer the reader to the paper \cite{R} which surveys the 
 study of 
 images of noncommutative polynomials in (mostly matrix) algebras   over the recent years.
 
 
 We will only be interested in the case where 
 $\A=M_n(F)$. 
 Then it may happen that $f(\A)=\{0\}$, in which case we call $f$ a {\em polynomial identity} of $\A$.
 It may also happen that 
 $f(\A)$ consists only of scalar matrices but $f(\A)\ne \{0\}$, in which case we call $f$ a 
{\em central polynomial} for $\A$. Further,  if $f$
is the sum of commutators in  $F\langle \mathcal X  \rangle$
and a polynomial identity, then all matrices in 
$f(\A)$ have trace zero. This explains why the question of representing a matrix $T$ as a linear combination of matrices from $f(\A)$ with $f$ any nonconstant polynomial makes sense only when $T$ has trace zero.

In \cite{BS} it was shown that
if $F=\C$ and
 $f\in  F\langle \mathcal X\rangle$  is neither a polynomial identity nor a central polynomial of $\A=M_n(F)$, then every trace zero matrix $T$ in $\A$ can be written as $T=A_1+A_2-A_3-A_4$ for some $A_i\in f(\A)$. It was also shown that for certain polynomials $f$ satisfying these assumptions, not every trace zero matrix $T$ is the difference of two elements from $f(\A)$. 
 The degrees of such polynomials $f$ must be larger than a certain number depending on the size of the matrices. The same comment can be made  for degrees of polynomial identities and central polynomials (see (I) below).
 That is why the problem presented in the introduction makes sense only for sufficiently large $n$.
 
The main goal of this paper is to prove the following theorem.

\begin{theorem}\label{mt}
Let $F$ be an algebraically closed field of characteristic $0$, 
  let $f\in F\langle \mathcal X\rangle $ be a polynomial of degree $m\ge 1$, and let $\A=M_n(F)$ with $n\ge 2$ and  $n \ge m-1$. Further, let
 $\alpha_1,\alpha_2,\alpha_3$ be nonzero elements from $F$ such that $\alpha_1+\alpha_2+\alpha_3=0$. Then every trace zero matrix $T$ in $\A$ can be written as
 $$T=\alpha_1 A_1+\alpha_2A_2+\alpha_3A_3$$ for some
 $A_1,A_2,A_3\in f(\A)$.
\end{theorem}

In the proof  we will need the following three results.
\begin{enumerate}
    \item[(I)] Let $f$ be a noncommutative polynomial of 
    degree $m\ge 1$. It is a standard fact that if $p$ is a positive integer such that $m<2p$, then $f$ cannot be a polynomial identity of $M_p(F)$.
    Using this fact, it can be easily shown that under the same assumption $f$ also  cannot be a  central polynomial for $M_p(F)$ (see, for example, \cite[Corollary 2.4]{F}). 
    
    \item[(II)] Let $F$ be an algebraically closed field of characteristic $0$, let $p$ be a prime,  
    and let $f$ be a polynomial which is neither a polynomial identity
nor a central polynomial of $M_p(F)$. Then
$f(M_p(F))$ contains a diagonal matrix $D$ with distinct eigenvalues  on the diagonal. Although  not stated as a theorem, this is clearly evident from the proof of \cite[Theorem 4.1]{BS}.
    \item[(III)] We will also use {\em Bertrand's postulate} which states that for every integer $r\ge 4$ there exists a prime number $p$ such that $r < p <2r-2$.
\end{enumerate}

\bigskip

\noindent {\em Proof of Theorem \ref{mt}.} We first claim that there exists
a prime number $p$ such that
$$\frac{n}{2} +1\le  p\le n.$$
 If $n=2$ then we take $p=2$,
and if $n=3$ or $n=4$ then we take $p=3$.
If $n\ge 5$, then
$[\frac{n+1}{2}] + 1\ge 4$ and so (III) yields the existence of a prime number $p$ such that
$$\Bigl[\frac{n+1}{2}\Bigr]+1 < p < 2\Bigl(\Bigl[\frac{n+1}{2}\Bigr]+1\Bigr)-2 = 2 \Bigl[\frac{n+1}{2}\Bigr],
$$
which readily implies that $p$ satisfies the desired condition.

Since
$$m\le n+1 \le 2(p-1)+1 = 2p-1$$
it follows from (I) that $f$ is neither a polynomial identity nor a central polynomial of $M_p(F)$.
Therefore, (II) tells us that 
$f(M_p(F))$ contains a diagonal matrix $D$ containing $p$ distinct diagonal entries. Identifying every matrix  $A\in M_p(F)$ with the matrix $\left[ \begin{smallmatrix} A & 0 \cr 0  & 0 \cr \end{smallmatrix} \right]\in M_n(F)$, we can consider $D$ as
a diagonal matrix belonging to $f(\A)$. Since
$$Sf(A_1,\dots,A_m)S^{-1} =
 f(SA_1S^{-1},\dots, SA_mS^{-1})$$
holds for all $A_i,S\in \A$ with $S$ invertible, $f(\A)$ also contains 
all matrices that are similar to $D$.
Observe that $D$ has either $p-1$ or $p$ distinct nonzero   diagonal entries  and all its other diagonal entries are $0$. Since $$\frac{n}{2} \le p-1 < p \le n,$$
the desired conclusion follows from  Theorem \ref{glavni}
(with either $q=p-1$ or $q=p$).
\hfill\(\Box\)

\bigskip

\noindent
{\bf Acknowledgment}. 
The authors would like to thank the referee for insightful comments.

\end{document}